\documentclass[10pt, notitlepage]{article}   

\usepackage{amsmath,amsthm,amsfonts}   

\usepackage{amscd}
\usepackage{amsfonts}
\usepackage{amssymb}
\usepackage{color}      
\usepackage{epsfig}
\usepackage{graphicx}           

\theoremstyle{plain}
\newtheorem{Thm}{Theorem}[section]
\newtheorem{Lem}[Thm]{Lemma}
\newtheorem{Crlr}[Thm]{Corollary}
\newtheorem{Prop}[Thm]{Proposition}
\newtheorem{Obs}[Thm]{Observation}

\theoremstyle{definition}

\theoremstyle{remark}

\def\finf{\mathop{{\rm I}\kern -.27 em {\rm F}}\nolimits}

\begin{document}

\title{Total Domination Value in Graphs}

\author{\bf Cong X. Kang\\
\small{Texas A\&M University-Galveston, Galveston, TX 77553, USA} \\
{\small\em kangc@tamug.edu}}

\maketitle

\date{}

\begin{abstract}
A set $D \subseteq V(G)$ is a \emph{total dominating set} of
$G$ if for every vertex $v \in V(G)$ there exists a vertex $u \in
D$ such that $u$ and $v$ are adjacent. A total dominating set of
$G$ of minimum cardinality is called a $\gamma_t(G)$-set.
For each vertex $v \in V(G)$, we define the \emph{total domination
value} of $v$, $TDV(v)$, to be the number of $\gamma_t(G)$-sets to which $v$
belongs. This definition gives rise to \emph{a local study of total domination}
in graphs. In this paper, we study some basic properties of the $TDV$ function; also, we derive explicit formulas for the $TDV$ of any complete n-partite graph, any cycle, and any path.
\end{abstract}


\section{Introduction}

Let $G = (V(G),E(G))$ be a simple, undirected, and nontrivial graph without isolated vertices. For $S \subseteq V(G)$, we denote
by $<\!S\!>$ the subgraph of $G$ induced by $S$. A set $D \subseteq V(G)$ is a \emph{total dominating set} (TDS) of
$G$ if for any $v \in V(G)$ there exists a $u \in
D$ such that $uv\in E(G)$. The \emph{total domination
number} of $G$, denoted by $\gamma_t(G)$, is the minimum
cardinality of a TDS in $G$; a TDS of $G$ of minimum
cardinality is called a $\gamma_t(G)$-set. The notion of total domination in
graphs was introduced by Cockayne et al.~\cite{EJ}. For a survey of total domination in graphs, see \cite{Henning}. For other
concepts in domination, refer to \cite{Dom1}. We generally follow \cite{CZ} for notation and graph theory terminology.\\

For each vertex $v \in V(G)$, we define the \emph{total domination value} of $v$, denoted by $TDV_{G}(v)$, to be the number of
$\gamma_t(G)$-sets to which $v$ belongs; we often drop $G$ when ambiguity is not a concern. We also define $\tau(G)$ to be the total number of
$\gamma_t(G)$-sets. Clearly, $0 \le TDV_G(v) \le \tau(G)$ for any $G$ and any $v\in G$. This definition gives rise to \emph{a local study of total domination} in graphs which is as natural as the notion of total domination itself, starting with the motivating problem of the five queens, as described by Cockayne et al.\,in~\cite{EJ}. A casual chess player is aware that it is important to control the center squares of the chessboard -- particularly
in the initial and middle phases of the game: thus, in a certain sense, center squares have greater (total) domination value; we'll take a look at a couple of ``miniature chess boards" at the end of next section. In any real-world situation which can be modeled by a graph and where (total) domination is of interest, the particular locations commanding high (total) domination values -- strategic high grounds,
if you will -- are obviously important. Though over a thousand papers have already been published on a
plethora of domination topics as of the late 1990's (see p.1 of~\cite{Henning}), a systematic local study of (total) domination is either new or not
well-known. However, in \cite{tdv}, Cockayne, Henning, and Mynhardt characterized the vertices in trees which attain extremal total domination values. In this paper, we study some basic properties of the $TDV$ function; we also derive explicit formulas for the $TDV$ of any complete n-partite graph, any cycle, and any path. For an analogous discussion on the $DV$ (\emph{domination value}) function, see~\cite{Yi}.\\


\section{Basic properties of $TDV$: upper and lower bounds}

For a vertex $v \in V(G)$, the \emph{open neighborhood} $N(v)$
of $v$ is the set of all vertices adjacent to $v$ in $G$, and the
\emph{closed neighborhood} of $v$ is the set $N[v]=N(v) \cup
\{v\}$. In this section, we consider the lower and upper bounds of
the $TDV$ function for a fixed vertex $v_0$ and for
$v \in N[v_0]$. If equality is obtained for a graph of some order in an inequality (the bound), we will say the bound is sharp.
We first make the following\\

\begin{Obs}\label{observation}
$\displaystyle \sum_{v \in V(G)} TDV_G(v) = \tau(G) \cdot
\gamma_t(G)$
\end{Obs}

To see this, list the vertices of each $\gamma_t(G)$-set in a row
--- forming a table with $\tau$ rows. $TDV(v)$ is the number of
appearances (possibly zero) the vertex $v$ makes in the table of
size $\tau$ by $\gamma_t(G)$.  \hfill \\

\begin{Obs}\label{isom invariant}
If there is an isomorphism of graphs carrying a vertex $v$ in $G$
to a vertex $v'$ in $G'$, then $TDV_{G}(v)=TDV_{G'}(v')$.
\end{Obs}

The $TDV$ function is obviously invariant under isomorphism. Many basic types of graphs, such as cycles and paths, admit obvious automorphisms. \\

\begin{Obs}\label{observation2}
Let $G$ be the disjoint union of two graphs $G_1$ and $G_2$. Then
$\gamma_t(G)=\gamma_t(G_1)+ \gamma_t(G_2)$ and $\tau(G)=\tau(G_1) \cdot \tau(G_2)$.
For $v \in V(G_1)$, $TDV_G(v)=TDV_{G_1}(v) \cdot \tau(G_2)$.
\end{Obs}

\begin{Prop}\label{upperbound1}
For a fixed $v_0 \in V(G)$, we have
$$\tau(G) \le \sum_{v \in N[v_0]} TDV_G(v) \le \tau(G) \cdot \gamma_t(G),$$ and both bounds are
sharp.
\end{Prop}

\begin{proof}
The upper bound follows from Observation \ref{observation}.
For the lower bound, note that every $\gamma_t(G)$-set $\Gamma$
must contain a vertex in $N[v_0]$: otherwise $\Gamma$ fails to totally dominate $v_0$.\\

For sharpness of the lower bound, take $v_0$ to be an end-vertex
of a path on 4 vertices. More generally, we can take $v_0$ to be
an end-vertex of a path on $4k$ vertices (see Theorem~\ref{theorem on paths} and Corollary~\ref{path
on 4k}). For sharpness of the upper bound, take as $v_0$ the
central vertex of a star. \hfill
\end{proof}

Remark: In fact, both the lower and upper bounds of Proposition \ref{upperbound1}
are achieved for a graph of order $n$ for any $n\geq 4$. Let $G_4$ be a path on $4$ vertices,
and we construct $G_n$ for $n\geq 5$ from $G_4$ by taking one support vertex $u$ of $G_4$,
$n-4$ new vertices, and draw one edge from $u$ to each of the $n-4$ new vertices.
For sharpness of the lower bound, take as $v_0$ any end-vertex. For sharpness of the upper bound, take as $v_0$ any support vertex.  \\

\begin{Obs}
If $s$ is a support vertex of $G$, then $\sum_{v \in N[s]} TDV(v) \ge 2 \tau(G)$,
since each $\gamma_t$-set must contain every support vertex $s$ and a neighbor of $s$.
More generally, the bound holds for any vertex $v$ for which $TDV(v)=\tau$.
\end{Obs}

\begin{Prop} \label{upperbound2}
For any $v_0 \in V(G)$, $$\sum_{v \in N[v_0]} TDV_G(v) \le \tau(G)
\cdot (1+ \deg_G(v_0)),$$ and the bound is sharp.
\end{Prop}

\begin{proof}
For each $v\in N[v_0]$, $TDV(v)\leq \tau(G)$, and the number of vertices in $N[v_0]$ equals $1+ \deg_{G}(v_0)$. Thus,\\
$$\sum_{v\in N[v_0]}TDV(v) \leq \sum_{v\in N[v_0]}\tau(G)=\tau(G)\!\!\!\sum_{v\in N[v_0]}1=\tau(G)(1+ \deg_{G}(v_0)).$$ \\
The upper bound is achieved for a graph of order $n$ for any $n\geq 5$. Let $G_5$ be a path on $5$ vertices,
and we construct $G_n$ for $n\geq 6$ from $G_5$ by taking one support vertex $u$ of $G_5$, $n-5$ new vertices,
and draw one edge from $u$ to each of the $n-5$ new vertices. To see the sharpness of the upper bound,
take as $v_0$ the vertex of degree two, which is the common neighbor of support vertices. \hfill
\end{proof}

\begin{figure}[htbp]
\begin{center}
\scalebox{0.5}{\input{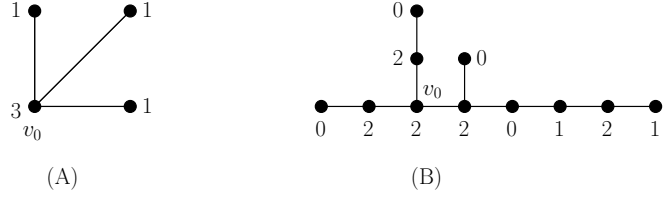}} \caption{Examples of
local total domination values and their upper
bounds}\label{figure1}
\end{center}
\end{figure}

We look at examples which compare the upper bounds of Proposition
\ref{upperbound1} and Proposition \ref{upperbound2}.
Let $v_0$ be the vertex of degree $3$ in graph (A) of Figure \ref{figure1}.
Then $\sum_{v \in N[v_0]} TDV(v)=6$. Note that $\tau =3$,
$\gamma_t =2$, and $\deg(v_0)=3$. Proposition \ref{upperbound1} yields the upper bound $\tau \cdot \gamma_t =3\cdot 2=6$,
which is sharp. But, the upper bound provided by Proposition \ref{upperbound2}
is $\tau (1+\deg(v_0))=3 \cdot (1+3)=12$, which is not sharp in this case. \\

Now, let $v_0$ be the vertex of degree $3$ and adjacent to
three support vertices, as labeled in graph (B) of Figure \ref{figure1}. Then $\sum_{v \in N[v_0]} TDV(v)=8$. Note that
$\tau =2$, $\gamma_t =6$, and $\deg(v_0)=3$. Proposition \ref{upperbound2} yields
the upper bound $\tau (1+ \deg(v_0))=2 \cdot (1+3)=8$, which is sharp.
But, the upper bound provided by Proposition \ref{upperbound1} is $\tau \cdot \gamma_t =2 \cdot 6=12$, which is not sharp in this case. \\

\begin{Prop}\label{subgraph-tau}
Let $H$ be a subgraph of $G$ with $V(H)=V(G)$. If
$\gamma_t(H)=\gamma_t(G)$, then $\tau(H) \le \tau(G)$ .
\end{Prop}

\begin{proof}
By the first assumption, every TDS for $H$ is a TDS for $G$. By
$\gamma_t(H)=\gamma_t(G)$, it's guaranteed that every TDS of
minimum cardinality for $H$ is also a TDS \emph{of minimum
cardinality for $G$}.
\end{proof}

Next, recall the following

\begin{Thm}(\cite{EJ}) \label{ej1}
If $G$ is a connected graph with $n \ge 3$ vertices, then
$\gamma_t(G) \le \frac{2n}{3}$.
\end{Thm}

\begin{Crlr} \label{up1}
If $G$ is a connected graph with $n \ge 3$ vertices, then
$$1 \le \tau(G) \le {n\choose \lfloor \frac{n}{2}\rfloor}$$
where ${n\choose i}$ is the binomial coefficient. Both bounds are
sharp.
\end{Crlr}

\begin{proof}
Notice that $\tau(G) \ge 1$. We will show the upper bound. Since
$\frac{2n}{3} \ge \lfloor \frac {n}{2} \rfloor$, by Theorem \ref{ej1}, we have
$$\tau(G) \le \max \left\{{n\choose 2}, {n\choose 3}, \cdots, {n\choose\lfloor\frac{2n}{3}\rfloor}\right\}
\le {n\choose \lfloor \frac{n}{2}\rfloor},$$

\noindent where the last inequality easily follows from, say, the
``Pascal's triangle". For sharpness of the lower bound, consider a path on $4k$ vertices or
an extended star (obtained from a star with at least three vertices by joining a path of length one
to each end-vertex of the star).
For sharpness of the upper bound, one may take $G$ to be $K_3$, $K_4$, or $K_5$ --
complete graphs on $3$, $4$, or $5$ vertices, respectively. (Notice that the upper bound is not achieved for any other graph.) \hfill
\end{proof}

Let $\Delta(G)$ denote the maximum degree of $G$, and $\bar{G}$ the complement of $G$. We recall
\begin{Thm} (\cite{EJ}) \label{G, Gc}
If $G$ has $n$ vertices, no isolates, and $\Delta(G)
< n-1$, then $\gamma_t(G)+\gamma_t(\bar{G}) \le n+2$, with
equality if and only if $G$ or $\bar{G}=mK_2$.
\end{Thm}

\begin{Prop}
Let $G$ be a graph on $n=2m \ge 4$ vertices. If $G$ or $\bar{G}$
is $mK_2$, then $$TDV_G(v)+TDV_{\bar{G}}(v)=n-1.$$
\end{Prop}

\begin{proof}
Without loss of generality, assume $G=mK_2$ and label the vertices of $G$ by $1,\ldots,2m$. Further assume that the vertex $2k-1$ is
adjacent to the vertex $2k$, where $1 \le k \le m$. Clearly, $\tau (G)=1=TDV_G(v)$ for any $v\in V(G)$. \\

Now, consider $\bar{G}$ and the vertex labeled $1$ for ease of
notation. It's obvious that $\gamma_t(\bar{G})=2$, and $\{1,
\alpha\}$ as $\alpha$ ranges from $3$ to $2m$ enumerates all total
dominating sets containing the vertex $1$. Thus
$TDV_{\bar{G}}(1)=2m-2=n-2$. By relabeling the vertices, we see
that $TDV_{\bar{G}}(v)=n-2$ holds for any $v \in V(\bar{G})$.
Therefore, $TDV_G(v)+TDV_{\bar{G}}(v)=n-1$. \hfill
\end{proof}

\begin{Lem} \label{gamma-t 2}
For any graph $G$ with $\gamma_t(G)=2$, $TDV(v) \leq \deg(v)$ for any $v\in V(G)$.
\end{Lem}

\begin{proof}
For each $\gamma_t$-set $\Gamma$ containing $v$, the other member of $\Gamma$ must be a vertex in $N(v)$, and $|N(v)|=\deg(v)$. \hfill
\end{proof}

\begin{Prop} \label{degree n-1}
Let $G$ be a graph of order $n$ such that $\Delta (G)=n-1$. Then $\gamma_t(G)=2$ and $TDV(v) \le n-1$ for any $v \in V(G)$. Equality holds if and
only if $\deg(v)=n-1$.
\end{Prop}

\begin{proof}
If $\deg(v)=n-1$, then $\gamma_t=2$; this is because $\{v, w\}$, where $w \in N(v)$, is a $\gamma_t$-set.
Then Lemma \ref{gamma-t 2} gives $TDV(v) \le n-1$ for any $v \in V(G)$. The last assertion is clear.
\end{proof}

At this juncture, we should state that the seminal paper~\cite{EJ} by Cockayne et al. already contains the following

\begin{Thm}(\cite{EJ})
\begin{itemize}
\item[(i)] If $G$ has $n$ vertices and no isolates, then $\gamma_t(G) \le n-\Delta(G)+1$.
\item[(ii)] If $G$ is connected and $\Delta(G) < n-1$, then $\gamma_t(G) \le n-\Delta(G)$.
\end{itemize}
\end{Thm}

Our present focus is on total domination value; the need for coherence and being (by and large)
self-contained renders unavoidable some duplication of results in~\cite{EJ}.

\begin{Prop}\label{degree n-2}
Let $G$ be a graph of order $n$ such that $\Delta(G)=n-2$.
Then $\gamma_t(G)=2$ and $TDV(v) \leq n-2$ for any $v \in V(G)$. Further, if $\deg(v)=n-2$, then $TDV(v)=|N(w)|$ where $vw \notin E(G)$.
\end{Prop}

\begin{proof}
Let $\deg(v)=n-2$, so there's only one vertex $w$ such that $vw \notin E(G)$. Since $G$ is without isolated vertices and $\Delta(G)=n-2$ (and thus connected), $w$ is adjacent to at least one of the
vertices, say $z$, in $N(v)$. Clearly, $\{v, z\}$ is a $\gamma_t$-set; so $\gamma_t=2$.
Noticing $N(v)\cap N(w)=N(w)$, we see that the number of $\gamma_t$-sets containing $v$ is $|N(w)|$;
i.e., $TDV(v)=|N(w)|$. Also, Lemma \ref{gamma-t 2} implies that $TDV(v) \leq \deg(v) \leq n-2$ for any $v \in V(G)$. \hfill
\end{proof}

\begin{Prop}\label{tau gamma-t 2}
If $G$ has order $n$ with $\gamma_t(G)=2$ and $\Delta(G) \le n-2$,
then $\displaystyle \tau(G) \le {n\choose 2} - \left\lceil \frac{n}{2} \right\rceil$. This bound is sharp.
\end{Prop}

\begin{proof}
Since $\gamma_t(G)=2$, choosing a $\gamma_t(G)$-set is the same as choosing an edge of $G$. Thus,
\begin{equation}\label{tau}
\tau(G) \le {n\choose 2}-|E(\overline{G})|
\end{equation}
But $\displaystyle |E(\overline{G})| \ge \left\lceil\frac{n}{2}\right\rceil$:
Consider the minimum number of edges to delete from $K_n$ to get to $G$.
The deletion of one edge reduces the degree by one to a pair of vertices. To
ensure $\Delta(G) \le n-2$, a minimum of $\frac{n}{2}$ edge deletions
must be made if $n$ is even, and a minimum of
$\frac{n-1}{2}+1=\frac{n+1}{2}$ edge deletions must be made if $n$
is odd. Thus, by inequality~(\ref{tau}), $\tau(G) \le {n\choose 2} - \lceil
\frac{n}{2} \rceil$. \\

To see the sharpness of this bound, let $G$ be the $(n-2)$-regular graph for any even $n \ge 4$.
Each vertex $v\in V(G)$ may be paired with any $w\in N(v)$ to form a $\gamma_t(G)$-set,
since any $u\notin\{v,w\}$ is adjacent to either $v$ or $w$. Thus $TDV(v)=\deg(v)=n-2$.
Observation \ref{observation} gives $n(n-2)=2\tau(G)$; i.e., $\tau(G)=\frac{n(n-2)}{2}$,
which equals ${n\choose 2} - \left\lceil \frac{n}{2} \right\rceil$. \hfill
\end{proof}

\begin{figure}[htbp]
\begin{center}
\scalebox{0.5}{\input{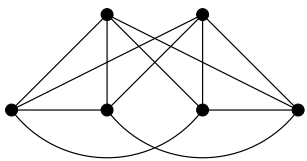}} \caption{Example showing the sharpness
of the upper bound for Proposition \ref{tau gamma-t 2} when $n=6$}\label{figure2}
\end{center}
\end{figure}

\begin{Thm}\label{theorem on Delta(G)=n-3}
Let $G$ be a graph of order $n \ge 4$ and $\Delta (G)=n-3$. Fix a vertex $v$ with $\deg(v)= \Delta(G)$.
\begin{itemize}
\item [(i)] If $G$ is disconnected, then $\gamma_t(G)=4$ and $TDV(v)=n-3$.
\item [(ii)] If $G$ is connected, then $\gamma_t(G)=2$ with $TDV(v) \le n-3$ or $\gamma_t(G)=3$ with $TDV(v)
\le (\frac{n-3}{2})^2+2(n-4)$.
\end{itemize}
\end{Thm}

\begin{proof}
Since $\deg(v)=n-3$, there are two vertices, say $\alpha$ and
$\beta$, such that $v \alpha, v \beta \not\in E(G)$. We consider
four cases.\\

\begin{figure}[htbp]
\begin{center}
\scalebox{0.5}{\input{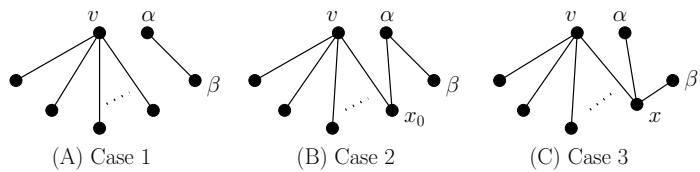}} \caption{Cases 1, 2, and 3
when $\Delta(G)=n-3$}\label{max1}
\end{center}
\end{figure}

\emph{Case 1. Neither $\alpha$ nor $\beta$ is adjacent to any
vertex in $N[v]$:} Since $G$ has no isolates, $\alpha \beta \in
E(G)$ (see (A) of Figure \ref{max1}). Let $G^{\prime}=<\!V(G)-
\{\alpha, \beta\}\!>$. Then $\deg_{G'}(v)=n-3$ with
$|V(G^{\prime})|=n-2$.
By Proposition~\ref{degree n-1}, $\gamma_t(G')=2$ and $TDV_{G'}(v)=n-3$; Observation \ref{observation2}, with $\gamma_t(<\!\{\alpha,\beta\}\!>)=2$ and $\tau(<\!\{\alpha,\beta\}\!>)=1$, yields $\gamma_t(G)=4$ and $TDV_G(v)=n-3$.\\

\emph{Case 2. Exactly one of $\alpha$ and $\beta$ is adjacent to a
vertex in $N(v)$:} Without loss of generality, assume that
$\alpha$ is adjacent to a vertex, say $x_0$, in $N(v)$. Since $G$
has no isolates, $\alpha \beta \in E(G)$ and $\alpha$ is a support
vertex of $G$. (See (B) of Figure \ref{max1}.) Thus $\alpha$
belongs to every $\gamma_t(G)$-set. First suppose $\gamma_t(G)=2$.
Then there exists a common neighbor of $v$ and
$\alpha$, say $w$, such that $\{\alpha, w\}$ is a $\gamma_t(G)$-set. Since $\alpha$ is a
support vertex of $G$ (i.e., $TDV(\alpha)=\tau(G)$) and $v \alpha
\not\in E(G)$, $TDV(v)=0$. Next consider $\gamma_t(G)>2$.
Since $\{v, x_0, \alpha\}$ is a $\gamma_t(G)$-set, $\gamma_t(G)=3$
and $TDV(v) =|N(v) \cap N(\alpha)| \le n-4$, where $n \ge 5$. (If $|N(v) \cap N(\alpha)|=n-3$,
then $\deg(\alpha)=n-2$ because $\beta \in N(\alpha)-N(v)$, contradicting the assumption that $\deg(v)=\Delta(G)=n-3$.)\\

\emph{Case 3. There exists a vertex in $N(v)$, say $x$, that is
adjacent to both $\alpha$ and $\beta$:} Notice that $n \ge 6$ in
this case, since $vx, \alpha x, \beta x \in E(G)$ and
$\deg(v)=\Delta (G)$ (see (C) of Figure \ref{max1}). Since $\{v,
x\}$ is a $\gamma_t(G)$-set, $\gamma_t(G)=2$. By Lemma \ref{gamma-t 2}, $TDV(v) \le \deg(v) \le n-3$ for any $v \in V(G)$.\\

\begin{figure}[htbp]
\begin{center}
\scalebox{0.5}{\input{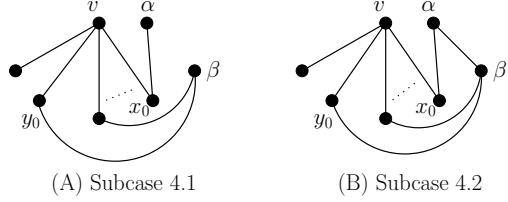}} \caption{Subcases 4.1
and 4.2 when $\Delta(G)=n-3$}\label{max2}
\end{center}
\end{figure}

\emph{Case 4. There exist vertices in $N(v)$ that are adjacent to
$\alpha$ and $\beta$, but no vertex in $N(v)$ is adjacent to both
$\alpha$ and $\beta$:} Let $x_0 \in N(v) \cap N(\alpha)$ and $y_0
\in N(v) \cap N(\beta)$. We consider two
subcases.\\

\emph{Subcase 4.1. $\alpha \beta \not\in E(G)$ (see (A) of Figure
\ref{max2}):} First, assume $\gamma_t(G)=2$. This is
possible when $\{x_0, y_0\}$ is a $\gamma_t(G)$-set satisfying
$x_0y_0 \in E(G)$ and $N(x_0) \cup N(y_0)=V(G)$. In this case, $n
\ge 7$. Notice that there's no $\gamma_t(G)$-set containing $v$
when $\gamma_t(G)=2$: no $u\in N(v)$ is adjacent to both $\alpha$ and $\beta$.
Thus $TDV(v)=0$. Second, assume
$\gamma_t(G)>2$. Since $\{v, x_0, y_0\}$ is a $\gamma_t(G)$-set,
$\gamma_t(G) =3$. Noticing that every $\gamma_t(G)$-set contains
a vertex in $N(\alpha)$ and a vertex in $N(\beta)$ and that $N(\alpha) \cap N(\beta) =\emptyset$, we see
$$TDV_G(v) = |N(\alpha)| \cdot |N(\beta)| \le
\left(\frac{|N(\alpha)|+|N(\beta)|}{2}\right)^2 \le
\left(\frac{n-3}{2}\right)^2,$$ where the first inequality is the
arithmetic-geometric mean inequality ($\frac{a+b}{2} \geq \sqrt{ab}$ for $a, b \geq 0$).

\emph{Subcase 4.2. $\alpha \beta \in E(G)$ (see (B) of Figure
\ref{max2}):} First assume $\gamma_t(G)=2$. This is
possible as in Subcase 4.1 or when $\{x_0, \alpha\}$ (resp. $\{y_0, \beta\}$) is a
$\gamma_t(G)$-set with $N(x_0) \cup N(\alpha)=V(G)$ (resp. $N(y_0)
\cup N(\beta)=V(G)$). In this case, $n \ge 6$. Notice that there's no $\gamma_t(G)$-set
containing $v$ with $\gamma_t(G)=2$, as in Subcase 4.1. Thus
$TDV_G(v)=0$. Second, consider $\gamma_t(G) > 2$. Since
$\{v, x_0, y_0\}$ is a $\gamma_t(G)$-set, $\gamma_t(G)=3$. The
number of \mbox{$\gamma_t(G)$-sets} containing $v$, but containing neither $\alpha$ nor $\beta$, is bounded above by
$(|N(\alpha)|-1)(|N(\beta)|-1) \le (\frac{n-3}{2})^2$ (by the arithmetic-geometric mean inequality).
The number of $\gamma_t(G)$-sets containing $v$ and $\alpha$ (resp. $\beta$) is bounded above
by $|N(\alpha)|-1 \le n-4$ (resp. $|N(\beta)|-1 \le n-4$). Thus $TDV(v) \le (\frac{n-3}{2})^2+2(n-4)$.
\end{proof}

Remark: From our proof of Theorem~\ref{theorem on Delta(G)=n-3} emerges a noteworthy fact
that one may have $TDV(v)=0$ even though $\deg_{G}(v)=\Delta(G)$: see graph in Figure~\ref{zeroTDV-max-deg}; also notice that the addition of the edge $\alpha \beta$ will not necessitate a change in its caption. See \cite{tdv} for a characterization of extremal $TDV$ values for trees.

\pagebreak

\begin{figure}[h]
\begin{center}
\scalebox{0.5}{\input{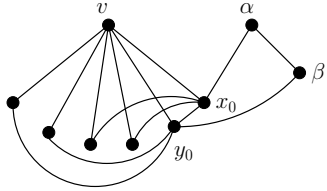}} \caption{$TDV(v)=0$, $\deg(v)=n-3$
is the unique maximum, and $\{x_0, y_0\}$ is the unique $\gamma_t$-set.}\label{zeroTDV-max-deg}
\end{center}
\end{figure}

\begin{figure}[htpb]
\begin{center}
\scalebox{0.5}{\input{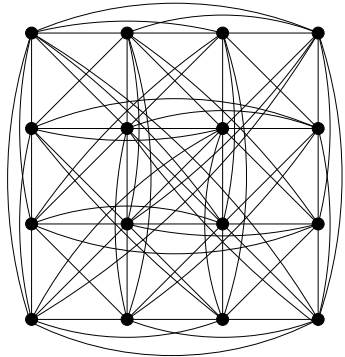}} \caption{The graph
induced by the queen's movement on the $4 \times
4$ chess board.}\label{Queen4}
\end{center}
\end{figure}

Illustration: As promised in the introduction, let's consider the
queen's movement on (for simplicity) a $3\!\times\!3$
``chessboard" and (separately) on a $4\!\times\!4$ ``chessboard".
Let's assume that the queen can move, as usual, any number of
squares horizontally, vertically, or diagonally (so long as there
are no other chess pieces lying in its way). Figure~\ref{Queen4} shows the graph induced by the queen's
movement on the $4\!\times\!4$ chessboard: two vertices are adjacent if and only if the queen -- the lone chess
piece on board -- can go between the corresponding squares
(represented by vertices) in one move. Observe that $\gamma_t=2$ for both graphs. On the $3\!\times\!3$ chessboard (not shown),
the center square has $TDV=8$ and each of the eight squares on the
periphery has $TDV=4$. On the $4\!\times\!4$ chessboard, the four
center squares each has $TDV=3$ and each of the twelve squares on
the periphery has $TDV=1$. Checking the forgoing claims is a straightforward matter: Take the $4\!\times\!4$ chessboard, for example. It's easily seen that $\gamma_t=2$; it follows that $TDV(v)=|\{w\in N(v): V(G)=N(w)\cup N(v)\}|$. By symmetry, it suffices to check only the $TDV$ of three vertices. (One can draw just the vertices for simplicity; pick a vertex $v_0$ to consider, and then consider each vertex $w_i$ dominated by a queen at $v_0$. Then $\{v_0, w_i\}$ forms a $\gamma_t$-set if $N(v_0) \cup N(w_i)=V(G)$.)\\


\section{Total domination value in complete $n$-partite graphs}

For a \emph{complete} $n$-partite graph $G$ -- where $n\geq 2$,
let $V(G)$ be partitioned into $n$-partite sets $V_1$, $V_2$, $\ldots$, $V_n$,
and let $a_i=|V_i|\geq 1$ for each $1\leq i \leq n$.

\begin{Prop}
Let $G$ be a complete $n$-partite graph with notation as specified above.
Then
$$\tau(G)=\frac{1}{2}\left[\left(\sum_{i=1}^{n} a_i\right)^2 - \sum_{i=1}^{n} a_i^2\right] \hskip .2in
\mbox{ and } \hskip .2in TDV(v)=\left(\sum_{i=1}^{n} a_i\right)-a_j \mbox{ if } v \in V_j.$$
\end{Prop}

\begin{proof}
Any two vertices from different partite sets form a
$\gamma_t(G)$-set, so $\gamma_t(G)=2$.
If $v \in V_j$, then $TDV(v)=\deg(v)=(\sum_{i=1}^{n} a_i)-a_j$. By Observation \ref{observation}, we have the following

\begin{eqnarray*}
\sum_{v \in V(G)} TDV(v) = \tau(G)\gamma_{t}(G) \label{n-partite} \hskip .5in
\Longrightarrow \hskip .5in \sum_{j=1}^{n}\sum_{v\in V_j} TDV(v) =2 \tau(G), \\
\sum_{j=1}^{n}\left(a_j\sum_{i=1}^{n} a_i -a_{j}^{2}\right) =2\tau(G) \ \
\Longrightarrow \ \ \left(\sum_{i=1}^{n} a_i\right)\left(\sum_{j=1}^{n} a_j\right)-\sum_{j=1}^{n} a_j^2 =2\tau(G),
\end{eqnarray*}
and the formula claimed for $\tau(G)$ follows.\hfill
\end{proof}

\vspace{.2in}

If $a_i=1$ for each $i$, $1\leq i\leq n$, then $G=K_n$ is a complete graph.
\begin{Crlr}
If $G=K_n$, where $n \ge 2$, then $\tau(G)= {n \choose 2}$ and $TDV(v)=n-1$.
\end{Crlr}

If $n=2$, then $G=K_{a_1, a_2}$ is a complete bipartite graph.
\begin{Crlr}
If $G=K_{a_1, a_2}$, then $\tau(G)=a_1\!\!\cdot\!a_2$ and
\begin{equation*}
TDV(v)= \left\{
\begin{array}{ll}
a_2 & \mbox{ if } v \in V_1 \\
a_1 & \mbox{ if } v \in V_2 .
\end{array} \right.
\end{equation*}
\end{Crlr}


\section{Total domination value in cycles}

Let $C_n$ be a cycle on $n$ vertices, labeled $1$ through $n$ consecutively in
counter-clockwise order. Observe that, by symmetry (or vertex-transitivity),
$TDV$ must be constant on the vertices of $C_n$ for each $n$. We denote by
$TDM(G)$ the collection of all $\gamma_t(G)$-sets. For $n \ge 3$, recall (p.368, \cite{CZ})
\begin{equation*}
\gamma_t(C_n) = \left\{
\begin{array}{ll}
\ \frac{n}{2} & \mbox{ if } n \equiv 0  \mbox{ (mod 4)} \\
\lfloor \frac{n}{2} \rfloor +1 & \mbox{ otherwise .}
\end{array} \right.
\end{equation*}

\noindent {\bf Examples.} (a) $\gamma_t(C_4)=2$,
$TDM(C_4)=\{\{1, 2\}, \{2, 3\}, \{3, 4\}, \{4, 1\}\}$;
so $\tau(C_4)=4$ and $TDV(i)=2$ for each $i \in V(C_4)$.\\

(b) $\gamma_t(C_5)=3$, $TDM(C_5)=\{\{1, 2, 3\}, \{2, 3,4\}, \{3, 4,5\}, \{4, 5, 1\}, \{5, 1, 2\}\}$;
so $\tau(C_5)=5$ and $TDV(i)=3$ for each $i \in V(C_5)$.\\

\begin{Thm}\label{theorem on cycles} Let $n \geq 3$, then
\begin{equation*}
\tau(C_n) = \left\{
\begin{array}{lr}
\ 4 & \mbox{ if } n \equiv 0  \mbox{ (mod 4)} \\
\ n & \mbox{ if } n \equiv 1, 3  \mbox{ (mod 4)} \\
(\frac{n}{2})^2 & \mbox{ if } n \equiv 2  \mbox{ (mod 4)}
\end{array} \right.
\end{equation*}
\end{Thm}

\begin{proof}

First, let $n=4k$, where $k\geq 1$. Here $\gamma_t=2k$;
a $\gamma_t$-set $\Gamma$ comprises $k$ $P_2$'s and $\Gamma$
is fixed by the choice of the first $P_2$. There are exactly
two $\gamma_t$-sets containing the vertex $1$, namely
the $\gamma_t$-set containing the path $\{n,1\}$ and
the different $\gamma_t$-set containing the path $\{1,2\}$;
by symmetry, there must be two $\gamma_t$-sets omitting the vertex $1$;
thus $\tau=4$. (Alternatively, since $TDV(1)=2$ and $\sum_{1}^{4k} 2= 2k\tau$
-- observations~\ref{observation}~and~\ref{isom invariant}, we also have $\tau=4$.)\\

Second, let $n=4k+1$, where $k\geq 1$. Here $\gamma_t=2k+1$;
a $\gamma_t$-set $\Gamma$ comprises $(k-1)$ $P_2$'s and one $P_3$.
And $\Gamma$ is fixed by the choice of the single $P_3$.
Choosing a $P_3$ is the same as choosing its middle vertex;
so the number of choices is simply $4k+1$. Thus $\tau=4k+1=n$.\\

Third, let $n=4k+2$, where $k\geq 1$. Here $\gamma_t=2k+2$;
a $\gamma_t$-set $\Gamma$ is constituted in exactly one of
the following three ways: 1) $\Gamma$ comprises $(k-1)$ $P_2$'s
and one $P_4$; 2) $\Gamma$ comprises $(k-2)$ $P_2$'s
and two $P_3$'s; 3) $\Gamma$ comprises $(k+1)$ $P_2$'s. \\

\emph{Case 1) $<\Gamma> \cong (k-1)P_2 \cup P_4$:} Note
that $\Gamma$ is fixed by the choice of the single $P_4$.
Choosing a $P_4$ is the same as choosing its initial vertex
in the counter-clockwise order. Thus $\tau=4k+2=n$.\\

\emph{Case 2) $<\Gamma> \cong (k-2)P_2 \cup 2 P_3$:} Note
that here $k\geq 2$ and $\Gamma$ is fixed by the placements
of the two $P_3$'s. There are $n=4k+2$ ways of choosing the first $P_3$,
as discussed. Consider the $P_{4k-5}$ (a sequence of $4k-5$ slots)
obtained as a result of cutting from $C_{4k+2}$ the $P_7$
centered about the first $P_3$. The initial vertex of
the second $P_3$ may be placed in the first slot of
any of the $\lceil\frac{4k-5}{4}\rceil=k-1$ subintervals
of the $P_{4k-5}$. As the order of selecting the
two $P_3$'s is immaterial, $\tau=\frac{(4k+2)(k-1)}{2}$.\\

\emph{Case 3) $<\Gamma> \cong (k+1)P_2$:} Note that,
since each $P_2$ totally dominates 4 vertices, there are
exactly two vertices, say $x$ and $y$, each of whom is
adjacent to two distinct $P_2$'s in $\Gamma$.
And $\Gamma$ is fixed by the placements of $x$ and $y$.
There are $n=4k+2$ ways of choosing $x$. Consider
the $P_{4k-3}$ (a sequence of $4k-3$ slots) obtained
as a result of cutting from $C_{4k+2}$ the $P_5$ centered about $x$.
Vertex $y$ may be placed in the first slot of
any of the $\lceil\frac{4k-3}{4}\rceil=k$ subintervals of
the $P_{4k-3}$. As the order of selecting the two vertices
$x$ and $y$ is immaterial, $\tau=\frac{(4k+2)k}{2}$. \\

Summing over the three disjoint cases
(note the second summand vanishes when $k=1$), we get \\
$$\hspace*{0in}\tau(C_n)=(4k+2)+\frac{(4k+2)(k-1)}{2}+\frac{(4k+2)k}{2}=(2k+1)^2=\left(\frac{n}{2}\right)^2.$$

Finally, let $n=4k+3$, where $k\geq 0$.
Here $\gamma_t=2k+2$; a $\gamma_t$-set $\Gamma$ comprises of
only $P_2$'s and is fixed by the placement of
the only vertex which is adjacent in two directions
(counter-clockwise and clockwise) to $P_2$(s) in $\Gamma$. Thus $\tau(C_n)=n$. \hfill
\end{proof}

\begin{Crlr}
Let $v\in V(C_n)$, then
\begin{equation*}
TDV(v) = \left\{
\begin{array}{lr}
\ 2 & \mbox{ if } n \equiv 0  \mbox{ (mod 4)} \\
\left\lfloor\frac{n}{2}\right\rfloor+1 & \mbox{ if } n \equiv 1, 3  \mbox{ (mod 4)} \\
\frac{n}{2} \cdot \frac{n+2}{4} & \mbox{ if } n \equiv 2  \mbox{ (mod 4)}
\end{array} \right.
\end{equation*}
\end{Crlr}

\begin{proof}
Use Theorem~\ref{theorem on cycles}, Observation~\ref{observation}, Observation~\ref{isom invariant}, and vertex-transitivity of $C_n$.
\end{proof}

\section{Total domination value in paths}

Let $P_n$ be a path on $n$ vertices, labeled $1$ through $n$ consecutively. Since $P_n$ is $C_n$ with an edge (but no vertices) deleted,
$\gamma_t(P_n)\geq \gamma_t(C_n)$. On the other hand, for $n\geq 4$,
there is a $\gamma_t$-set $\Gamma$ of $C_n$ omitting a pair of adjacent
vertices -- making $\Gamma$ a $\gamma_t$-set of $P_n$. Thus $\gamma_t(P_n)=\gamma_t(C_n)$; explicitly stated, for $n \ge 2$,
\begin{equation*}
\gamma_t(P_n) = \left\{
\begin{array}{ll}
\frac{n}{2} & \mbox{ if } n \equiv 0  \mbox{ (mod 4)} \\
\lfloor \frac{n}{2} \rfloor +1 & \mbox{ otherwise .}
\end{array} \right.
\end{equation*}

\noindent {\bf Examples.} (a) $\gamma_t(P_4)=2$, $TDM(P_4)=\{\{2,
3\}\}$; so $\tau(P_4)=1$ and
\begin{equation*}
TDV(i)= \left\{
\begin{array}{ll}
1 & \mbox{ if } i =2,3 \\
0 & \mbox{ if } i =1,4 .\\
\end{array} \right.
\end{equation*}

(b) $\gamma_t(P_6)=4$, $TDM(P_6)=\{\{ 2, 3, 4, 5\}, \{1, 2,4,5\},
\{1,2,5,6\}, \{2,3, 5, 6\}\}$; so $\tau(P_6)=4$, and
\begin{equation*}
TDV(i)= \left\{
\begin{array}{ll}
2 & \mbox{ if } i =1,3,4,6 \\
4 & \mbox{ if } i =2,5 .\\
\end{array} \right.
\end{equation*}

Remark: Note that $\tau(P_n) \le \tau(C_n)$ for $n \ge 3$ by
Proposition \ref{subgraph-tau}. In fact, we have

\begin{Thm}\label{theorem on paths}
For $n \ge 2$,
\begin{equation*} \tau(P_n) = \left\{
\begin{array}{lr}
\ 1 & \mbox{ if } n \equiv 0  \mbox{ (mod 4)} \\
\ \lfloor \frac{n}{4} \rfloor & \mbox{ if } n \equiv 1  \mbox{ (mod 4)} \\
\ (\lfloor \frac{n}{4} \rfloor +1)^2 & \mbox{ if } n \equiv 2
\mbox{ (mod 4)}\\
\ \lfloor \frac{n}{4}\rfloor + 2 & \mbox{ if } n \equiv 3 \mbox{
(mod 4)}
\end{array} \right.
\end{equation*}
\end{Thm}

\begin{proof}
First, let $n=4k$, where $k \ge 1$. Then $\gamma_t=2k$ and a
$\gamma_t$-set $\Gamma$ comprises $k$ $P_2$'s. In this case, every
two adjacent vertices in $\Gamma$ totally dominates four vertices,
and no vertex of $P_{4k}$ is totally dominated by more than one
vertex. Thus none of the end-vertices of $P_n$ belong to any
$\Gamma$, which contains and is fixed by $\{2,3\}$; hence $\tau=1$.\\

Second, let $n=4k+1$, where $k \ge 1$. Here $\gamma_t=2k+1$; a
$\gamma_t$-set $\Gamma$ comprises $(k-1)$ $P_2$'s and one $P_3$.
Since each component with cardinality $c$ in $<\Gamma>$ totally
dominates $c+2$ vertices, no end-vertices belong to any $\Gamma$.
Note that $\Gamma$ is fixed by the placement of the single $P_3$, and
there are $\lfloor\frac{n}{4}\rfloor$ slots where the $P_3$ may be placed; so $\tau=\lfloor\frac{n}{4}\rfloor$.\\

Third, let $n=4k+2$, where $k \ge 0$. Here $\gamma_t=2k+2$; a
$\gamma_t$-set $\Gamma$ is constituted in exactly one of the
following three ways: 1) $\Gamma$ comprises $(k-1)$ $P_2$'s and
one $P_4$; 2) $\Gamma$ comprises $(k-2)$ $P_2$'s and two $P_3$'s;
3) $\Gamma$ comprises $(k+1)$ $P_2$'s.\\

\emph{Case 1) $<\Gamma> \cong (k-1)P_2 \cup P_4$, where $k \ge
1$:} Note that $\Gamma$ is fixed by the placement of the single
$P_4$, and none of the end-vertices belong to any $\Gamma$, as each component with cardinality $c$ in $<\Gamma>$ totally
dominates $c+2$ vertices. The $P_4$ may be placed in one of the $\lfloor\frac{n}{4}\rfloor=k$ slots.\\

\emph{Case 2) $<\Gamma> \cong (k-2)P_2 \cup 2P_3$, where $k \ge
2$:} Note that again, none of the end-vertices belong to any $\Gamma$, and $\Gamma$ is fixed by the placements of the two
$P_3$'s into the $k$ available slots. Thus $\tau={k\choose 2} =\frac{k(k-1)}{2}$.\\

\emph{Case 3) $<\Gamma> \cong (k+1)P_2$, where $k \ge 0$:} A
$\Gamma$ containing both end-vertices of the path is unique (no vertex is doubly dominated). The
number of $\Gamma$ containing exactly one of the end-vertices (one doubly dominated vertex) is
$2{k \choose 1}=2k$. The number of $\Gamma$ containing none of the end-vertices (two doubly dominated vertices) is ${k \choose 2}= \frac{k(k-1)}{2}$. Thus
$\tau=1+2k+\frac{k(k-1)}{2}$.\\

Summing over the three disjoint cases (note that $\tau=1$ when $k=0$), we get
$$\tau(P_n)=k+\frac{k(k-1)}{2}+\left(1+2k+\frac{k(k-1)}{2} \right)=k^2+2k+1=(k+1)^2= \left( \left\lfloor \frac{n}{4}\right\rfloor+1 \right)^2.$$

Finally, let $n=4k+3$, where $k \ge 0$. Here $\gamma_t=2k+2$, and
a $\gamma_t$-set $\Gamma$ comprises $k+1$ $P_2$'s. There is no $\Gamma$ containing both end-vertices of $P_n$. The
number of $\Gamma$'s containing exactly one of the end-vertices (no doubly dominated vertex) of
the path is two. The number of $\Gamma$'s containing neither of the
end-vertices (one doubly dominated vertex) is $k$. Summing the two disjoint cases, we have $\tau(P_n)=k+2=\lfloor \frac{n}{4}\rfloor+2$.
\end{proof}

For the total domination value of a vertex on $P_n$, note that $TDV(v)=TDV(n+1-v)$ for \mbox{$1\leq v\leq n$} as $P_n$ admits the obvious automorphism carrying $v$ to $n+1-v$. More precisely, we have the classification result which follows. First, as an immediate consequence of Theorem \ref{theorem on paths}, we have

\begin{Crlr} \label{path on 4k}
Let $v \in V(P_{4k})$, where $k \ge 1$. Then
\begin{equation*}
TDV(v)= \left\{
\begin{array}{ll}
0 & \mbox{ if } v \equiv 0,1  \mbox{ (mod 4)} \\
1 & \mbox{ if } v \equiv 2,3  \mbox{ (mod 4)} .
\end{array} \right.
\end{equation*}
\end{Crlr}

\begin{Prop}\label{path on 4k+1}
Let $v \in V(P_{4k+1})$, where $k \ge 1$. Write $v=4q+r$, where
$0 \le r < 4$. Then, noting $\tau(P_{4k+1})=k$, we have
\begin{equation}\label{tdv of path on 4k+1}
TDV(v)= \left\{
\begin{array}{ll}
q & \mbox{ if } v \equiv 0  \mbox{ (mod 4)} \\
0 & \mbox{ if } v \equiv 1 \mbox{ (mod 4) } \\
k-q & \mbox{ if } v \equiv 2 \mbox{ (mod 4) }\\
k & \mbox{ if } v \equiv 3 \mbox{ (mod 4) }
\end{array} \right.
\end{equation}
\end{Prop}

\begin{proof}
We prove by induction on $k$. The base, $k=1$ case, is easily
checked. Assume that (\ref{tdv of path on 4k+1}) holds for
$G=P_{4k+1}$ and consider $G'=P_{4k+5}$. First, notice that each
$\Gamma$ of the $k$ $\gamma_t$-sets of $G$ induces a
$\gamma_t$-set $\Gamma'=\Gamma\cup\{4k+3, 4k+4\}$ of $G'$.
Additionally, $G'$ has the $\gamma_t$-set $\Gamma^*$ which
contains and is determined by $\{4k+2, 4k+3, 4k+4\}$. The presence
of $\Gamma^*$ implies that $TDV_{G'}(v)=TDV_{G}(v)+1$ for $v\leq
4k+1$ and $v\equiv 2 \mbox{ or } 3 \pmod 4$. Still, $TDV_{G'}(v)=TDV_G(v)$ for $v \le 4k+1$ and $v \equiv 0 \mbox{ or } 1 \pmod 4$ .Clearly,
$TDV_{G'}(4k+2)=1, TDV_{G'}(4k+3)=k+1, TDV_{G'}(4k+4)=k+1, \mbox{ and }
TDV_{G'}(4k+5)=0$.
\end{proof}

Remark: The proofs (we have) of the $TDV$ formulas in the propositions from this point onward are all inductive and have rather similar arguments. Thus, to avoid undue repetitiveness, we will offer ``sketches of proofs" and leave some details to the readers. \\

\begin{Prop}
Let $v \in V(P_{4k+2})$, where $k \ge 0$. Write $v=4q+r$, where $0 \le r < 4$. Then, noting $\tau(P_{4k+2})=(k+1)^2$, we have
\begin{equation*} TDV(v)= \left\{
\begin{array}{ll}
(k+1)q & \mbox{ if } v \equiv 0 \mbox{ (mod 4) } \\
(k+1)(q+1) & \mbox{ if } v \equiv 1 \mbox{ (mod 4) } \\
(k+1)(k+1-q) & \mbox{ if } v \equiv 2 \mbox{ (mod 4) } \\
(k+1)(k-q) & \mbox{ if } v \equiv 3 \mbox{ (mod 4) } \\
\end{array} \right.
\end{equation*}
\end{Prop}

\textit{Sketch of Proof:}
Let $\Gamma$ be a $\gamma_t(P_{4k+2})$-set for $k \ge 0$. We
consider three cases.\\

\textit{Case 1) $<\Gamma> \cong (k-1)P_2 \cup P_4$, where $k \ge
1$:} Denote by $TDV'(v)$ the number of such $\Gamma$'s containing
$v$. Noting $\tau=k$ in this case, we have
\begin{equation}\label{path 4k+2 case1}
TDV'(v)= \left\{
\begin{array}{ll}
q & \mbox{ if } v \equiv 0, 1  \mbox{ (mod 4)} \\
k-q & \mbox{ if } v \equiv 2,3 \mbox{ (mod 4) }
\end{array} \right.
\end{equation}
A proof proceeds by induction on $k$ is similar to the proof of Proposition~\ref{path on 4k+1}: No end-vertex belongs to any $\gamma_t$-set. There is one $\gamma_t$-set of $P_{4(k+1)+2}$ which contains $\{4k+2, 4k+3, 4k+4, 4k+5\}$, while there are $k$ $\gamma_t$-sets of $P_{4(k+1)+2}$ derived from $\gamma_t$-sets of $P_{4k+2}$ which do not contain $\{4k+2, 4k+3\}$. \\

\textit{Case 2) $<\Gamma> \cong (k-2)P_2 \cup 2P_3$, where $k \ge
2$:} Denote by $TDV''(v)$ the number of such $\Gamma$'s containing
$v$. Noting $\tau={k \choose 2}$ in this case and setting ${a
\choose b}=0$ when $a<b$, we have
\begin{equation}\label{path 4k+2 case2}
TDV''(v)= \left\{
\begin{array}{ll}
\vspace{.15in}
\sum_{j=1}^{q} (k-j) & \mbox{ if } v \equiv 0  \mbox{ (mod 4)} \\
\vspace{.1in}
{q \choose 2} & \mbox{ if } v \equiv 1 \mbox{ (mod 4) } \\
\vspace{.1in}
{k-q \choose 2} & \mbox{ if } v \equiv 2 \mbox{ (mod 4) }\\
{k \choose 2}-{q \choose 2} & \mbox{ if } v \equiv 3 \mbox{ (mod 4) }
\end{array} \right.
\end{equation}

A proof proceeds by induction on $k$: No end-vertex belongs to any
$\gamma_t(P_{4k+2})$-set. The base, $k=2$ case, is easily checked.
Assume that (\ref{path 4k+2 case2}) holds for $G=P_{4k+2}$ and
consider $G'=P_{4k+6}$. First, notice that each $\Gamma$ of the
${k \choose 2}$ $\gamma_t(G)$-set induces a $\gamma_t(G')$-set
$\Gamma'=\Gamma \cup \{4k+4, 4k+5\}$. Additionally, there are
$k={k+1 \choose 2}-{k \choose 2}$ $\gamma_t(G)$-sets $\Gamma^*$
which contain $\{4k+3, 4k+4, 4k+5\}$. Having the sets $\Gamma^*$ implies that $TDV''_{G'}(4k+2)=TDV''_G(4k+2)$ and
$TDV''_{G'}(v)=TDV''_G(v)+TDV_{P_{4k+1}}(v)$ for $v \leq 4k+1$,
where $TDV_{P_{4k+1}}(v)$ is given in formula (\ref{tdv of path on 4k+1}).
Therefore, for $v \leq 4k+2$, we have
\begin{equation}
TDV''_{G'}(v)= \left\{
\begin{array}{ll}
\vspace{.1in}
\left[\sum_{j=1}^{q} (k-j)\right] +q = \sum_{j=1}^{q}(k+1-j)& \mbox{if } v \equiv 0  \mbox{ (mod 4)}\\
\vspace{.1in}
{q \choose 2} & \mbox{if } v \equiv 1 \mbox{ (mod 4)}\\
\vspace{.1in}
{k-q \choose 2} + k-q= {k+1-q \choose 2}& \mbox{if } v \equiv 2 \mbox{ (mod 4)}\\
{k \choose 2}-{q \choose 2} +k={k+1 \choose 2}-{q \choose 2}&
\mbox{if } v \equiv 3 \mbox{ (mod 4)}
\end{array} \right.
\end{equation}
To finish this case, one separately checks that $TDV''_{G'}(4k+3)={k+1 \choose 2}-{k \choose 2}=k$,
$TDV''_{G'}(4k+4)=TDV''_{G'}(4k+5)={k+1 \choose 2}$, and
$TDV''_{G'}(4k+6)=0$.\\

\textit{Case 3) $<\Gamma> \cong (k+1)P_2$, where $k \ge 0$:}
Denote by $TDV'''(v)$ the number of such $\Gamma$'s containing
$v$. First, suppose both end-vertices belongs to each $\Gamma$;
denote by $TDV_1'''(v)$ the number of such $\Gamma$'s containing
$v$. Then there's a unique $\Gamma$ in this case, and we have
\begin{equation}\label{Case 3-1}
TDV'''_{1}(v)= \left\{
\begin{array}{ll}
1 & \mbox{ if } v \equiv 1, 2  \mbox{ (mod 4)} \\
0 & \mbox{ if } v \equiv 0,3 \mbox{ (mod 4) }
\end{array} \right.
\end{equation}

Second, suppose exactly one end-vertex belongs to each $\Gamma$;
denote by $TDV_2'''(v)$ the number of such $\Gamma$'s containing
$v$. Write $v = 4q+r$, $0 \le r < 4$. Then, noting $\tau=2k$ in
this case, we have
\begin{equation}\label{Case 3-2}
TDV'''_{2}(v)= \left\{
\begin{array}{ll}
q & \mbox{ if } v \equiv 0 \mbox{ (mod 4)} \\
k+q & \mbox{ if } v \equiv 1 \mbox{ (mod 4)} \\
2k-q & \mbox{ if } v \equiv 2 \mbox{ (mod 4)} \\
k-q & \mbox{ if } v \equiv 3 \mbox{ (mod 4)}
\end{array} \right.
\end{equation}
Notice each of $k$ $\gamma_t$-sets of $P_{4k+2}$ containing the left end-vertex is paired with
vertices $4k+4$ and $4k+5$ in $P_{4k+6}$; each of $k$ $\gamma_t$-sets of $P_{4k+2}$
containing the right end-vertex is paired with vertices $4k+5$ and $4k+6$ in $P_{4k+6}$.
Additionally, a $\gamma_t$-set of $P_{4k+2}$ containing both left and right end-vertices of $P_{4k+2}$
may be paired with vertices $4k+4$ and $4k+5$ in $P_{4k+6}$; there is also a $\gamma_t$-set in $P_{4k+6}$
containing vertices $4k+3, 4k+5, \mbox{ and } 4k+6$ (making $4k+4$ the sole doubly-dominated vertex).
Induction on $k$ readily verifies the claimed formula. \\

Third, suppose no end-vertex belongs to $\Gamma$; denote by
$TDV_3'''(v)$ the number of such $\Gamma$'s containing $v$. Write
$v = 4q+r$, $0 \le r < 4$. Then, noting $\tau={k \choose 2}$ in
this case, we have
\begin{equation} \label{Case 3-3}
TDV'''_{3}(v)= \left\{
\begin{array}{ll}
\vspace{.1in}
{q \choose 2} & \mbox{ if } v \equiv 0 \mbox{ (mod 4)} \\
\vspace{.1in}
{k \choose 2}-{k-q \choose 2} & \mbox{ if } v \equiv 1 \mbox{ (mod 4)} \\
\vspace{.1in}
{k \choose 2}-{q \choose 2} & \mbox{ if } v \equiv 2 \mbox{ (mod 4)} \\
{k-q \choose 2} & \mbox{ if } v \equiv 3 \mbox{ (mod 4)}
\end{array} \right.
\end{equation}
Notice each of the ${k \choose 2}$ $\gamma_t$-sets of $P_{4k+2}$ containing neither
end-vertex is paired with vertices $4k+4$ and $4k+5$ in $P_{4k+6}$. Additionally,
each of the $k$ $\gamma_t$-sets of $P_{4k+2}$ containing the right end-vertex of $P_{4k+2}$
may be paired with vertices $4k+4$ and $4k+5$ in $P_{4k+6}$ (making $4k+3$ one of the two doubly-dominated vertices).
Induction on $k$ again readily verifies the claimed formula.\\

Summing over the three disjoint cases (\ref{Case 3-1}), (\ref{Case
3-2}), and (\ref{Case 3-3}) for $<\Gamma> \cong (k+1)P_2$ -- i.e., $\displaystyle TDV'''(v)=\sum_{i=1}^{3}TDV'''_{i}(v)$, we have
\begin{equation}\label{path 4k+2 case3}
TDV'''(v)=\left\{
\begin{array}{ll}
\vspace{.08in}
q + {q\choose 2}& \mbox{if } v \equiv 0 \mbox{ (mod 4)}\\
\vspace{.08in}
1+k+q+{k \choose 2}-{k-q \choose 2}& \mbox{if } v \equiv 1 \mbox{ (mod 4)}\\
\vspace{.08in}
1+2k-q+{k \choose 2}-{q \choose 2}& \mbox{if } v \equiv 2 \mbox{ (mod 4)}\\
\vspace{.08in}
k-q+ {k-q \choose 2}& \mbox{if } v \equiv 3 \mbox{ (mod 4)}
\end{array} \right.
\end{equation}
Now, sum over (\ref{path 4k+2 case1}), (\ref{path 4k+2 case2}), and (\ref{path 4k+2 case3}) -- namely $TDV(v)=TDV'(v)+TDV''(v)+TDV'''(v)$ -- to reach the formula claimed in this proposition. \hfill $\square$\\

\begin{Prop}
Let $v \in V(P_{4k+3})$, where $k \ge 0$. Write $v=4q+r$, where $0 \leq r < 4$. Then, noting $\tau(P_{4k+3})=k+2$, we have
\begin{equation*}
TDV(v)= \left\{
\begin{array}{ll}
0 & \mbox{ if } v \equiv 0 \mbox{ (mod 4) } \\
q+1 & \mbox{ if } v \equiv 1 \mbox{ (mod 4) } \\
k+2 & \mbox{ if } v \equiv 2 \mbox{ (mod 4) } \\
k+1-q & \mbox{ if } v \equiv 3 \mbox{ (mod 4) }
\end{array} \right.
\end{equation*}
\end{Prop}

\textit{Sketch of Proof:}
Let $\Gamma$ be a $\gamma_t$-set with $k \geq 0$. Note that no $\Gamma$ contains both end-vertices of $P_{4k+3}$.
We consider two cases. First, suppose $\Gamma$ contains exactly one end-vertex, and denote by $TDV'(v)$ the number of such $\Gamma$'s containing $v$. We have
\begin{equation*}\label{path on 4k+3 case1}
TDV'(v)= \left\{
\begin{array}{ll}
0 & \mbox{ if } v \equiv 0 \mbox{ (mod 4) }\\
1 & \mbox{ if } v \equiv 1,3 \mbox{ (mod 4) } \\
2 & \mbox{ if } v \equiv 2 \mbox{ (mod 4) }
\end{array} \right.
\end{equation*}
Next, assume $\Gamma$ contains no end-vertices -- so $k\geq 1$, and denote by $TDV''(v)$ the number of such $\Gamma$'s containing $v$. Writing
$v=4q+r$ for $0 \le r < 4$, we have
\begin{equation}\label{path on 4k+3 case2}
TDV''(v)= \left\{
\begin{array}{ll}
0 & \mbox{ if } v \equiv 0 \mbox{ (mod 4) } \\
q & \mbox{ if } v \equiv 1 \mbox{ (mod 4) } \\
k & \mbox{ if } v \equiv 2 \mbox{ (mod 4) } \\
k-q & \mbox{ if } v \equiv 3 \mbox{ (mod 4) }
\end{array} \right.
\end{equation}
There are $k$ $\gamma_t$-sets in this case, and formula~(\ref{path on 4k+3 case2}) can be proved by induction on $k$ as in the proof of Proposition~\ref{path on 4k+1}.
Now, $TDV(v)=TDV'(v)+TDV''(v)$, which is as claimed. \hfill $\square$ \\

\bigskip

\textit{Acknowledgement.} The author greatly appreciates Eunjeong Yi for much valuable help in drawing the figures and in deriving the formulas contained in this paper. The author also likes to thank the referee for a correction and some helpful suggestions which improved the paper.\\


\begin{thebibliography}{99}

\bibitem{CZ} G. Chartrand and P. Zhang, {\it Introduction to Graph
Theory.} McGraw-Hill, Kalamazoo, MI (2004).


\bibitem{EJ} E. J. Cockayne, R. M. Dawes, and S. T. Hedetniemi,
Total domination in Graphs. \textit{Networks} {\bf{10}} (1980),
211-219.

\bibitem{tdv} E. J. Cockayne, M. A. Henning, and C. M. Mynhardt,
Vertices contained in all or in no minimum total dominating set of a tree. \textit{Discrete Math.} {\bf{260}} (2003),
37-44.

\bibitem{Dom1} T. W. Haynes, S. T. Hedetniemi, and P. J. Slater,
\textit{Fundamentals of Domination in Graphs,} Mercel Dekker, New
York (1998).

\bibitem{Henning} M. A. Henning, A survey of selected recent results on total
domination in graphs. \textit{Discrete Math.} {\bf{309}}, Issue 1,
(2009), 32-63.

\bibitem{Yi} E. Yi, Domination Value in Graphs, \textit{preprint}

\end{thebibliography}
\end{document}